\documentclass{amsart}
\usepackage{amssymb,latexsym,hyperref}

\newtheorem{thrm}{Theorem}[section]
\newtheorem{lem}[thrm]{Lemma}

\theoremstyle{definition}

\theoremstyle{remark}

\numberwithin{equation}{section}

\def\Z{{\mathbb Z}}

\def\Aut{{\rm Aut}}

\def\Cay{{\rm Cay}}

\begin{document}

\title{Groups with elements of order 8 do not have the DCI property}{}

\author{Ted Dobson}
\address{University of Primorska, UP IAM, Muzejski trg 2, SI-6000 Koper, Slovenia and, University of Primorska, UP FAMNIT, Glagolja\v{s}ka 8, SI-6000 Koper, Slovenia}
\email{ted.dobson@upr.si}

\author{Joy Morris}           
\address{Department of Mathematics and Computer Science\\
University of Lethbridge\\
Lethbridge, AB. T1K 3M4}    
\email{joy.morris@uleth.ca}
\thanks{The second author was supported by the Natural Science and Engineering Research Council of Canada (grant RGPIN-2017-04905).}

\author{Pablo Spiga}
\address{Dipartimento di Matematica e Applicazioni, University of Milano-Bicocca,\newline
Via Cozzi 55, 20125 Milano, Italy}
\email{pablo.spiga@unimib.it}{}



\begin{abstract}
Let $k$ be odd, and $n$ an odd multiple of $3$. We prove that $C_k \rtimes C_8$ and $(C_n \times C_3)\rtimes C_8$ do not have the Directed Cayley Isomorphism (DCI) property. When $k$ is also prime, $C_k \rtimes C_8$ had previously been proved to have the Cayley Isomorphism (CI) property. To the best of our knowledge, the groups $C_p \rtimes C_8$ (where $p$ is an odd prime) are only the second known infinite family of groups that have the CI property but do not have the DCI property. This also shows that no group with an element of order $8$ has the DCI property.
\end{abstract}

\maketitle

\section{Introduction}

Let $\Gamma_1=\Cay(G,S)$ be a Cayley digraph on $G$. We say that the digraph $\Gamma_1$ has the Directed Cayley Isomorphism (DCI) property if whenever $\Gamma_2=\Cay(G,T)\cong \Gamma_1$, there is some $\alpha \in \Aut(G)$ such that $S^\alpha=T$. Analogously, we say that the graph $\Gamma$ has the Cayley Isomorphism (CI) property if whenever $\Gamma_2=\Cay(G,T)\cong \Gamma_1$, there is some $\alpha \in \Aut(G)$ such that $S^\alpha=T$. For graphs, we must have $S=S^{-1}$ and $T=T^{-1}$.

We say that the group $G$ has the DCI property if every Cayley digraph on $G$ has the DCI property, and that $G$ has the CI property if every Cayley graph on $G$ has the CI property. Observe that for a group, having the DCI property is stronger than having the CI property. 

In the 1970s, Babai~\cite{Babai} proved a criterion involving only permutation groups that allows us to determine whether or not a group $G$ has the DCI property: this happens if every $2$-closed group has at most one conjugacy class of regular subgroups isomorphic to $G$. 

Throughout this paper, $A$ will denote a group that is isomorphic to either $C_k$ or $C_k \times C_3$, where $k$ is odd. Our main result is the following:

\begin{thrm}\label{main}
Let $A$ be a group that is isomorphic to either $C_k$ or $C_k \times C_3$, where $k$ is odd. Let $R =A \rtimes \langle r\rangle$, where $|r|=8$ and $r^{-1}ar=a^{-1}$ for every $a \in A$.

Then $R$ does not have the DCI property.
\end{thrm}

This result is not new, but has never been published explicitly. The fact that the groups $A \rtimes C_8$ do not have the DCI property can be deduced from an old result of Babai and Frankl \cite{BF} in which they show that a quotient of a CI group by a characteristic subgroup is CI.  Although the Babai-Frankl proof is written for graphs rather than digraphs, their techniques apply to either situation. Since $C_8$ is not a DCI group and $A$ is the unique maximal subgroup of odd order and therefore characteristic in $A \rtimes C_8$, this means that $A \rtimes C_8$ is not DCI. However, $C_8$ is a CI group, and when $k$ is square-free and not prime, the CI-status of $A \rtimes C_8$ remains open. 

When $k$ is prime (and $A$ is cyclic), the groups $A \rtimes C_8$ have been proved to have the CI property by Li, Lu, and P\'alfy \cite{LLP}. As far as we are aware, this makes $C_p \rtimes C_8$ (where $p$ is an odd prime) only the second known infinite family of groups that have the CI property but do not have the DCI property. The other family is $C_p \rtimes C_3$, where $p$ is prime and $3 \mid p-1$. These groups were shown by the first author to be CI~\cite[Theorem 21]{Ted}.\footnote{A subsequent, self-contained proof of this also appears in~\cite{LLP}.} That they are not DCI groups follows from work by Li showing that they are not $k$-DCI for most values of $k$~\cite[Theorem 1.3]{Li1999}.

It has been unfortunately common for work on the CI problem and the DCI problem to be imprecise in terminology, so that some papers report that a group has the CI property when the proof actually shows the DCI property, or more importantly that a group is not CI when in fact the proof only shows that it is not a DCI group. Some of the confusion that has arisen through this lack of consistency in the terminology has been partially addressed in some recent papers~\cite{MM,JM}. However, those papers do not touch at all on $A \rtimes C_8$. 

There is a description that first appeared in the work of Li, Lu, and P\'alfy \cite{LLP} and has since been updated via \cite{DMS}, of groups that may have the CI property. For many of these, their status remains open. No equivalent description has been published for the DCI property. Groups not matching the description cannot be DCI since they are not CI, but groups that match the description need not be DCI even if they are CI. In the Li-Lu-P\'{a}lfy description (as updated), $A \rtimes C_8$ is the only family of groups that have elements of order $8$, whose CI status is unknown. Since the groups $A\rtimes C_8$ are not in fact DCI groups, there are $2$-closed groups that contain more than one conjugacy class of regular subgroups isomorphic to $A \rtimes C_8$.  This has a significant impact on the proof techniques that may be required to prove that the groups $A\rtimes C_8$ are CI (if, in fact, they are). So it seemed to us important to write this note clarifying what we currently know about these groups. 

We were able to produce a short and direct proof of Theorem~\ref{main}.
Our proof includes a new lemma that may be useful in other contexts, showing that a permutation group that has a regular subgroup of index $2$ is always $2$-closed. 

\section{The proof}

We will define two regular permutation groups $R_1$ and $R_2$ that are isomorphic to $A \rtimes C_8$, and show that $R_1$ and $R_2$ are not conjugate in $G^{(2)}$ (the 2-closure of $G$), where $G=\langle R_1, R_2\rangle$.
Using Babai's criterion \cite{Babai}, this implies that $C_k \rtimes C_8$ does not have the DCI property.

Let the underlying set $\Omega$ of cardinality $8|A|$  be written as $\Z_8 \times \Z_k \times \Z_r$, where $r=1$ if $A\cong C_k$, and $r=3$ if $A \cong C_3 \times C_k$. When adding or subtracting in a coordinate, the operation is performed in $\Z_k$ or $\Z_r$ (as appropriate).
\begin{itemize}
\item Define $\tau_1= ((0\ 1\ 2\ 3\ 4\ 5\ 6\ 7),(1),(1))$;
\item Define $\tau_2= ((0\ 1\ 6\ 7\ 4\ 5\ 2\ 3),(1),(1))$;
\item Define $\rho_1$ by $$(i,j,\ell)^{\rho_1}=\begin{cases}(i,j+1,\ell)& \text{$i$ even}\\(i,j-1,\ell)& \text{$i$ odd.}\end{cases}$$
\item Define $\rho_2$ by $$(i,j,\ell)^{\rho_1}=\begin{cases}(i,j,\ell+1)& \text{$i$ even}\\(i,j,\ell-1)& \text{$i$ odd.}\end{cases}$$
\item Take $R_1=\langle \tau_1, \rho_1, \rho_2\rangle$ and $R_2=\langle \tau_2, \rho_1, \rho_2\rangle$.
\end{itemize}

It is straightforward to calculate that the group $H=\langle \tau_1,\tau_2\rangle$ has order $16$ and any of its elements can be written as $\tau_2^\ell h^\epsilon$, where $\ell \in \Z_8$, $h=\left( (1\ 5)(3\ 7),(1),(1)\right)$, and $\epsilon \in \{0,1\}$. It is  also straightforward to see that $R_1=\langle \tau_1, \rho_1, \rho_2\rangle \cong A \rtimes C_8\cong R_2=\langle \tau_2, \rho_1,\rho_2\rangle$. We need to show two things:
\begin{enumerate}
\item there is no $g \in G$ such that $R_2^g=R_1$; and
\item $G$ is $2$-closed.
\end{enumerate}
We show these in the subsequent results.

%\begin{lem}\label{right size}
%For $G=\langle \tau_1,\tau_2,\rho\rangle$ as defined above, $|G|=16|A|$.
%\end{lem}
%\begin{proof}
%Let $g \in \Stab_G((0,0,0))$, so $(0,0,0)^g=(0,0,0)$. Since $\tau_1$ and $\tau_2$ normalise $\langle\rho_1, \rho_2\rangle$, any element of $G$ can be written as $g=\tau\rho_1^i\rho_2^j$ for some $\tau \in H=\langle \tau_1, \tau_2\rangle$. Since $$(0,0,0)^g=(0,0,0)^{\tau\rho_1^i\rho_2^j}=(0,0,0)$$ and $\rho_1^i\rho_2^j$ does not affect the first coordinate, $\tau$ must fix the first coordinate. Also, since $\tau$ does not affect the second or third coordinates, $\rho_1^i$ must fix the second coordinate and $\rho_2^j$ the third. Therefore
%$i \equiv 0 \pmod{k}$, $j \equiv 0 \pmod{r}$. This implies $g=\tau \in H$. 
%
%We know that $|H |=16$, so we have $$g=\tau \in \Stab_H((0,0,0))=\left \langle\left( (1\ 5)(3\ 7),(1),(1)\right)\right\rangle.$$ Thus $$\Stab_G((0,0,0))=\Stab_H((0,0,0))=\left \langle\left( (1\ 5)(3\ 7),(1),(1)\right)\right\rangle.$$ This implies $|G|=16|A|$.
%\end{proof}

\begin{lem}\label{not conj}
For $R_1, R_2$ as defined above and $G=\langle R_1, R_2\rangle$, there is no $g \in G$ such that $R_2^g=R_1$.
\end{lem}
\begin{proof}
Since $\tau_1$ and $\tau_2$ normalise $\langle\rho_1, \rho_2\rangle$, any element of $G$ can be written as $g=\tau\rho_1^i\rho_2^j$ for some $\tau \in H=\langle \tau_1, \tau_2\rangle$ and some $i \in \Z_k$, $j \in \Z_r$. Additionally, as discussed above, we may write $\tau=\tau_2^\ell h^\epsilon$, where $h=\left( (1\ 5)(3\ 7),(1),(1)\right)$, $\ell \in \Z_8$, and $\epsilon \in \Z_2$.  So $g=\tau_2^\ell h^\epsilon\rho_1^i\rho_2^j $.

Now if $R_2^g=R_1$ we must have $\tau_2^g \in R_1$. Since $\tau_2^g$ has order $8$, this implies $\tau_2^g \in \langle \tau_1\rangle$ (the Sylow $2$-subgroup of $R_1$). 

We have $$\tau_2^g=\tau_2^{\tau_2^\ell h^\epsilon\rho_1^i\rho_2^j}=\tau_2^{h^\epsilon\rho_1^i\rho_2^j}.$$
Now we calculate $\tau_2^h=\tau_2^5 \in R_2$; since $\rho_1, \rho_2 \in R_2$ we must have $\tau_2^g$ being an element of order $8$ in the Sylow $2$-subgroup of $R_2$. But none of these elements is in $R_1$. Therefore for every $g \in G$, $R_2^g \neq R_1$.
\end{proof}

We prove a more general lemma that implies the final piece we need.

\begin{lem}\label{index 2}
Let $G$ be a group with a regular subgroup $R$ of index $2$. Then $G$ is $2$-closed.
\end{lem}

\begin{proof}
Since $R$ is acting regularly, we can label the points of the underlying set with the elements of $R$. Let $z \in G^{(2)}$, with $z \neq 1$. We will show that $z \in G$. Since $R$ is regular, we may multiply $z$ by an element of $R$ if necessary to assume without loss of generality that $1^z=1$.

Since $R$ has index $2$ in $G$, there is some nonidentity element $g \in G$ of order $2$ that acts as an automorphism of $R$ (so fixes $1$), such that $G=R\langle g\rangle$.

Suppose that $a \in R$ with $a^z \neq a$. By the definition of $2$-closure, there is some $h \in G$ such that $(1^h,a^h)=(1,a)^h=(1,a)^z=(1,a^z)$. Since $g$ is the only nontrivial element of $G$ that fixes $1$, we must have $h=g$. Thus $a^z=a^g$.

Now suppose that $b \in R$ with $b^z=b$. Note that there is a unique element of $G$ that fixes $b$ and is not the identity, namely $b^{-1}gb$. By the definition of $2$-closure, there is some $h \in G$ such that $(b,a)^h=(b,a)^z=(b,a^g)\neq (b,a)$, so $h$ fixes $b$ but does not fix $a$. The only nontrivial element of $G$ that fixes $b$ is $b^{-1}gb$, so $h=b^{-1}gb$.

Now $$(b,a^g)=(b,a)^z=(b,a)^{b^{-1}gb}=(b,a^{b^{-1}gb})=(b,(ab^{-1})^gb).$$ Since $g$ is an automorphism of $R$, this is $(b,a^g(b^g)^{-1}b)$. Therefore $(b^g)^{-1}b=1$, implying $b^g=b$. So $b^z=b=b^g$.

We have now shown that for any $c \in R$, whether or not $c^z=c$ we can conclude that $c^z=c^g$. Thus $z=g \in G$. Since $z$ was an arbitrary element of $G^{(2)}$, $G$ is $2$-closed. 
\end{proof}

Putting these together produces our main result.

\begin{proof}[Proof of Theorem~\ref{main}]
Let $\tau_1$, $\tau_2$, $\rho_1$, $\rho_2$, $R_1$, $R_2$, and $G$ be as defined above. Then $G$ contains two regular subgroups ($R_1$ and $R_2$) isomorphic to $R$. By Lemma~\ref{index 2}, $G$ is $2$-closed. By Lemma~\ref{not conj}, $R_1$ and $R_2$ are not conjugate in $G=G^{(2)}$. Therefore by Babai's criterion~\cite{Babai}, $R$ does not have the DCI property.
\end{proof}

\bibliography{References}{}
\bibliographystyle{amsplain}

\end{document}